\documentclass{article}

\usepackage{microtype}
\usepackage{graphicx}
\usepackage{booktabs}
\usepackage{hyperref}

\usepackage[accepted]{icml2026}

\usepackage[utf8]{inputenc}
\usepackage[T1]{fontenc}
\usepackage{amsmath,amssymb,amsthm,bbm,paralist,comment,mathtools,accents,xcolor,tikz,amsfonts,nicefrac,url}

\newcommand{\expect}[1]{\mathbb{E}\left[ #1 \right]}

\newcommand{\prob}[1]{\mathbb{P}\left( #1 \right)} 
\newcommand{\cprob}[2]{\mathbb{P}\left( #1 \middle\vert #2 \right)}

\newcommand{\norm}[1]{\left\Vert #1 \right\Vert}
\newcommand{\oo}[1]{\left( #1 \right)}

\newcommand{\cc}[1]{\left[ #1 \right]}

\newtheorem{theorem}{Theorem}
\newtheorem{lemma}{Lemma}

\newtheorem{definition}{Definition}
\newtheorem{corollary}{Corollary}

\icmltitlerunning{Full Conformal Prediction under Stochastic Non-Conformity Measure}

\begin{document}

\twocolumn[
  \icmltitle{Full Conformal Prediction under\texorpdfstring{\\}{ }Stochastic Non-Conformity Measure}

  \begin{icmlauthorlist}
    \icmlauthor{Thanawat Sornwanee}{stanford}
  \end{icmlauthorlist}

  \icmlaffiliation{stanford}{Stanford University}
  \icmlcorrespondingauthor{Thanawat Sornwanee}{}
  \icmlkeywords{Conformal prediction, stochastic algorithms, permutation invariance}

  \vskip 0.3in
]

\printAffiliationsAndNotice{}

\begin{abstract}
    The theory of full conformal prediction uses deterministic non-conformity measure, but modern usage of full conformal prediction often relies on machine learning training, making stochasticity inevitable. A simple sufficient condition of almost sure permutation invariance of the non-conformity measure can be too restrictive, so many have suggested the relaxation to permutation in distribution as a condition for full conformal prediction validity. We, however, show that this commonly known condition is actually insufficient. We then provide a correct sufficient condition: \emph{Conditional Independence \& Permutation Invariance in Distribution}, which encompasses several stochastic settings that may be used in machine learning.
\end{abstract}

\section{Full conformal prediction}

    Given the dataset of 
    \begin{align*}
        \left(\cc{z_i=(x_i, y_i)}_{i=1}^N, x_{N+1}\right),
    \end{align*}
    where $x_i$ is the input and $y_i$ is the output, we are interested in constructing a confidence bound on the unobserved output $y_{N+1}$. 

    A standard conformal prediction relies on the routine of retrieving conformal p-value from the rank of the non-conformity score. Afterwards, the conformal prediction returns a  confidence set
    \begin{align*}
        \left\{\hat{y}\in \mathcal{Y}: \text{conformal p-value}\oo{\cc{z_i}_{i=1}^N, \oo{x_{N+1}, \hat{y}}} \ge \alpha\right\}.
    \end{align*}
    
    In full conformal prediction setting, we are given a non-conformity score function $t: \mathcal{Z}^{N+1} \to \mathcal{S}$, where $\mathcal{S}$ is a totally preordered set endowed with a binary relation\footnote{which is reflexive, transitive, and complete.}$\succcurlyeq$, and is normally chosen to be $\oo{\mathbb{R}, \ge}$.~\citep{vovk2005algorithmic,vovk2016criteria} We the compute the score in leave-one-out manner, and then compute the conformal p-value from the rank of the score as outlined in the algorithm~\ref{alg:deterp}.

    \begin{algorithm}
    \caption{Conformal p-Value via Deterministic $t$}
    \label{alg:deterp}
    \begin{algorithmic}[1]
    \REQUIRE Full Data $z \in \mathcal{Z}^{N+1}$
    \FOR{$i = 1$ to $N+1$}
        \STATE Compute Non-Conformity Score
        \[
        s_i = t\oo{z_1, \dots, z_{i-1}, z_{i+1}, \dots, z_{N+1}, z_i}\]
    \ENDFOR
    \STATE Compute Conformal p-Value
    \[P = \frac{1}{N+1} \cc{1 + \sum_{i=1}^{N} \mathbf{1}_{S_{N+1} \succcurlyeq S_i}}\]
    \STATE \textbf{return} $P$
    \end{algorithmic}
    \end{algorithm}

    The result is that, if $Z = \cc{Z_i = (X_i, Y_i)}_{i=1}^{N+1}$ is exchangeable, and the non-conformity score function $t$ is permutation invariant in the first $N$ arguments, then the conformal p-value of the actual data $Z$ is a valid p-value\footnote{See appendix~\ref{appendix:pperm} for a formal definition.}. This implies that the actual output $Y_{N+1}$ is in the confidence set with respect to $\oo{\cc{Z_i}_{i=1}^N, X_{N+1}}$ with probability $\ge 1-\alpha$, where the randomness is from the stochasticity of $Z$ itself.

    In most cases, the function $t$ uses the first $N$ data entries, which is $\{z_j\}_{j \ne i}$,  symmetrically to train a regression model\footnote{If the training returns more estimators, such as variance estimator or quantile estimator, one can modify the score to incorporate them.~\cite{papadopoulos2002inductive, romano2019conformalized}} $\hat{f}_{-i}: \mathcal{X} \to \mathcal{Y}$, and use the loss $\norm{y_i - \hat{f}_{-i}(x)}$ to be the non-conformity score $S_i$.

    Note that the trained model $\hat{f}_{-i}$ does not have to be good or satisfy any conditions apart from that it training process must be deterministic and permutation invariant in the data input $\{z_j\}_{j \ne i}$. For example, a ridge regression or machine learning with fixed initialization and stochastic descent will satisfy this. 
    
    Permutation invariance seems to be natural to most modern machine learning algorithm, but deterministic training requirement is often violated. Modern applications of full conformal prediction often rely on a training of a machine learning model, which may involve stochasticity via stochastic gradient descent as well as random initialization.~\cite{lee2025leaveoneout,tailor2025approximating} Some may also employ different stochastic training model such as random forests.~\citep{linusson2014efficiency,gauraha2018conformalclassification}
    
    Apart from the stochasticity in the training, we will see that the stochasticity in the inference time is also not uncommon.  The emergence usage of modern LLM (large language model) as well as other generative model can have stochasticity in the inference such as from the API call itself.\footnote{Currently, the usage of LLM as well as other pretrained models in conformal prediction is more limited to a split conformal setting.~\citep{ravfogel2023conformal, quach2024conformal,su2024api,epstein2026llms} However, one can imagine the use of full conformal prediction by doing a fine-tuning steps, or other post-training alignment methods, with the input data. Though this has not been extensively studied} \cite{angelopoulos2022gentle} suggest adding a small white noise to each score can also be used for tie-breaking.

    This suggests that a better understanding of full conformal prediction under a stochastic non-conformity measure is crucial. 

\section{Stochasticity}
    
    Before discussing the prior works, we will provide some formulation of full conformal prediction under stochastic measure/score. This is when we replace the routine of computing conformal p-value from using a deterministic function $t: \mathcal{Z}^{N+1} \to \mathcal{S}$ to stochastic function(s) $T$. Note that we will assume that every function mentioned has to be measurable throughout.

    Since we may evaluate different non-conformity score with different randomness\footnote{For example, $\omega$ may dictate $3$ independently drawn random seed $\epsilon_1$, $\epsilon_2$, and $\epsilon_3$. The training process to get score $1$ may only utilize the seed $\epsilon_1$, while that of the score $2$ only uses $\epsilon_2$. Therefore, mathematically, we need to put the indicator for $T$.}, we then need to provide flexibility for the stochastic function of each coordinate to be different. Formally, we have the algorithmic random space $\oo{\Omega, \mathcal{F}, \mathbb{P}_{\text{alg}}}$, which will throughout be assumed to be independent from that generating data, and a vector of stochastic functions $T = \cc{T_i}_{i=1}^{N+1}$ where  $T_i: \mathcal{Z}^{N+1} \times \Omega \to \mathcal{S}$ for each $i \in \{1,2,\dots, N+1\}$. 
    
    The routine of computing non-conformity score becomes the algorithm~\ref{alg:stochp}. Note that, for a given $z$, the p-value $P$ still inherits randomness from the algorithm.
    \begin{algorithm}
    \caption{Conformal p-Value via Stochastic $T$}
    \label{alg:stochp}
    \begin{algorithmic}[1]
    \REQUIRE Full Data $z \in \mathcal{Z}^{N+1}$, Algorithmic Random Space $\oo{\Omega, \mathcal{F}, \mathbb{P}_{\text{alg}}}$
    \STATE Realize $\omega \sim \mathbb{P}_{\text{alg}}$
    \FOR{$i = 1$ to $N+1$}
        \STATE Compute Stochastic Non-Conformity Score
        \[S_i = T_i\oo{z_1, \dots, z_{i-1}, z_{i+1}, \dots, z_{N+1}, z_i; \omega}\]
    \ENDFOR
    \STATE Compute Stochastic Conformal p-Value
    \[P = \frac{1}{N+1} \cc{1 + \sum_{i=1}^{N} \mathbf{1}_{S_{N+1} \succcurlyeq S_i}}\]
    \STATE \textbf{return} $P$
    \end{algorithmic}
    \end{algorithm}

    The question of interest is then
    \begin{quote}
    \emph{What is a sufficient condition on the random functions $T = \cc{T_i}_{i=1}^{N+1}$ for the validity of the conformal prediction?}
    \end{quote}
    This is when the actual observation $Y_{N+1}$ is included in the confidence set with probability $\ge 1-\alpha$, where now the randomness comes from both the data $Z$ itself and the algorithm. This is equivalent to asking for a sufficient condition for the conformal p-value in the algorithm~\ref{alg:stochp} to be a valid p-value whenever the random input $Z$ is an exchangeable random vector. Formally, for any exchangeable $Z$, we require, for all $\alpha \ge 0$,
    \begin{align*}
        \prob{\text{conformal p-value}\oo{Z; \omega} \le \alpha}\le \alpha.
    \end{align*}
    
    \subsection{Prior Works}
    As discussed, the stochasticity in full conformal prediction is not uncommon, but the theoretical study is lacking. To the author's knowledge, only 5 papers explicitly mention about condition for stochastic conformal score. The mentions are brief: 3 of them are in footnotes. 4 claims are also incorrect.

    \paragraph{Incorrect Claims}
    
    The paper ``Conformal Prediction Beyond Exchangeability"\citep{barber2023conformal} has claimed that in a case where the stochastic $T$ is such that
    \begin{align*}
        &T_i(z_1, z_2, \dots, z_{i-1}, z_{i+1}, z_{i+2}, \dots, z_{N+1}, z_i; \omega)\\&= \oo{\mathcal{A}(z_1, z_2, \dots, z_{i-1}, z_{i+1}, z_{i+2}, \dots, z_{N+1}; \omega)}\oo{z_i}
    \end{align*}
    for some random algorithm $\mathcal{A}$ mapping $z_{-i} \in \mathcal{Z}^{N}$ to a function $\hat{\mu}: \mathcal{X} \to \mathcal{Y}$, which is assumed to be $\mathbb{R}$, then we only require
    \begin{align*}
        \mathcal{A}\oo{z; \cdot} \overset{\text{d}}{=} \mathcal{A}\oo{\cc{z_{\sigma(i)}}_{i=1}^N; \cdot}
    \end{align*}
    for all $z \in \mathcal{Z}^{N+1}$ and for all permutation $\sigma \in S_N$.\footnote{The footnote 2 of~\citep{barber2023conformal} states that ``If $\mathcal{A}$ is a randomized algorithm, then this equality is only required to hold in a distributional sense," where the original equality refers to $\mathcal{A}\oo{z} =\mathcal{A}\oo{\cc{z_{\sigma(i)}}_{i=1}^N}$, which is a commonly used sufficient condition in deterministic full conformal prediction.} The same permutation symmetry in distribution sense is also claimed to be a sufficient condition in the paper ``Training-conditional coverage for distribution-free predictive inference"\citep{bian2023training}\footnote{The subsubsection 2.1.3 ``A note on randomized algorithm" of~\citep{bian2023training} states that ``The background given above implicitly  the algorithm $\mathcal{A}$ as a deterministic function of the training data—that is, we view $\mathcal{A}$ as a function $\oo{(X_1, Y_1), \dots, (X_n, Y_n)} \mapsto \hat{\mu}$. In many settings, however, it is common to use a randomized regression algorithm-for instance, stochastic gradient descent. In this setting, we can formally view $\mathcal{A}$ as a function $\oo{(X_1, Y_1), \dots, (X_n, Y_n), \xi} \mapsto \hat{\mu}$, where the term $\xi$ introduces stochastic noise (effectively, a random seed). All the results described above hold for both the deterministic and randomized settings. (For results that assume $\mathcal{A}$ is symmetric, the symmetry condition (6) should be understood in the distributional sense-that is, the training data points are treated symmetrically with respect to the randomized training procedure. For example, for stochastic gradient descent, if data points are drawn uniformly at random during the training epochs, then symmetry is satisfied.)"}and by~\cite{lee2023distribution}\footnote{The footnote 12 of~\cite{lee2023distribution} mentions that ``The framework also allows for a randomized algorithm $\mathcal{A}$, in which case the symmetry condition is required to hold in a distribution sense."}.
    This condition implies that, for all $i,j \in \{1,2,\dots, T+1\}$, $z \in \mathcal{Z}^{N+1}$, $\sigma \in S_N$,
    \begin{align*}
        T_i(z;\cdot) \overset{\text{d}}{=} 
        T_j\oo{\cc{z_{\sigma(k)}}_{k=1}^N, z_{N+1}; \cdot},
    \end{align*}
    and $T_i = T_j$. Note that the equation can be further decomposed into two parts: 1.) within the same index, equality in distribution across permutation and 2.) equality in distribution across different indices. The latter is implied by $T_i=T_j$, but we keep it as two separate conditions since we will soon relax the equality.
    
    This condition is claimed to be a sufficient condition by~\cite{bai2024optimized}.\footnote{The footnote 1 of~\citep{bai2024optimized} mentions that ``In the general case where $\mathcal{V}$ is a randomized algorithm, we require ... Definition 4 to hold in a distributional sense, i.e., $=$ could be replaced by $\overset{\text{d}}{=}$," where the equality refers to the one shown here.} 

    However, all these 4 papers are not correct, and we will give out a simple counterexample in the section~\ref{section:distributionisnotdufficiebt}.

    \paragraph{Correct but Narrow Claim} \cite{lee2025full} seem\footnote{The proof of the theorem 2 of~\citep{lee2025full} states that ``The scoring function $V^{(k)}(\cdot)$ was trained invariant to the order of elements inside $E_j$ (this assumes that $V^{(k)}$ does not use external randomness during training; similar results can be attained for random training procedures by conditioning on the random seed)," $V^{(k)}$ is equivalent to $T_i$ in our notation.} to claim that a sufficient condition is when, in an almost sure sense,
    \begin{align*}
        T_i(z;\omega) = 
        T_j\oo{\cc{z_{\sigma(k)}}_{k=1}^N, z_{N+1}; \omega}.
    \end{align*}
    Similarly, this condition can be decomposed into two components: almost sure equality across indices, and almost sure equality among permutation in the same index.

    This is obviously sufficient but quite strong. For example, it cannot handle stochastic gradient descent in a usual sense.\footnote{We provide some examples where this is satisfied in the subsection~\ref{subsection:almostsure}.}

\section{Why equality in distribution is insufficient?}
\label{section:distributionisnotdufficiebt}

    Below we will show a simple example where the equality in distribution under permutation is satisfied for the training algorithm $\mathcal{A}$, so
    the equality in distribution under permutation of the first $N$ element is satisfied for the stochastic scoring functions $T$, but the the full conformal prediction is invalid.

    Consider the case where $(X_i,Y_i) \overset{\text{iid}}{\sim} \text{Unif}\left([0,1]^2\right)$, and $N=2$. Thus, the random vector $Z = (Z_1, Z_2, Z_3) \in \oo{[0,1]^2}^3$ is exchangeable.
    
    Let the random space be such that we can define two independent standard Brownian motions $\left\{B_t(\omega)\right\}_{t \in [0,1]}$ and $\left\{B'_t(\omega)\right\}_{t \in [0,1]}$, and we define a random scoring function for all $i \in \{1,2,3\}$ to be
    \begin{align*}
        T_i\left(z; \omega\right) := \oo{\mathcal{A}\left(z_1, z_2; \omega\right)}(z_3)
        :=
        B_{y_2}(\omega) + B'_{1-y_2}(\omega),
    \end{align*}
    for any $z \in \oo{[0,1]^2}^3$, thereby making
    \begin{align*}
        T\left(z; \cdot\right) \sim N(0,1),
    \end{align*}
    and similarly $\mathcal{A}$ is permutation invariant in distribution, since it always return a constant function whose value is distributed according to $N(0,1)$.

    The non-conformity scores from the algorithm~\ref{alg:stochp} conditioned on the data $Z$ are
    \begin{align*}
        \begin{cases}
            \left(S_1 \vert Z\right)(\cdot)  &= T\left(Z_2, Z_3, Z_1;\cdot\right)
            = B_{Y_3} + B'_{1-Y_3}
            \\
            \left(S_2 \vert Z\right)(\cdot) &= T\left(Z_1, Z_3, Z_2;\cdot\right)
            = B_{Y_3} + B'_{1-Y_3}\\
            \left(S_3 \vert Z\right)(\cdot)  &= T\left(Z_1, Z_2, Z_3;\cdot\right)
            = B_{Y_2} + B'_{1-Y_2}
        \end{cases},
    \end{align*}
    making the conformal p-value 
    \begin{align*}
        \left(P\middle\vert Z\right)(\cdot) &=
        \left(\frac{1+\mathbf{1}_{S_3 \ge S_1} + \mathbf{1}_{S_3 \ge S_2}}{3} \middle\vert Z\right)\\
        &=
        \frac{1}{3} + \frac{2}{3}
        \left(\mathbf{1}_{B_{Y_2} + B'_{1-Y_2} \ge B_{Y_3} + B'_{1-Y_3}} \middle\vert Z\right).
    \end{align*}
    Under the case when $Y_2 \ne Y_3$, which happens almost surely, we have that $B_{Y_2} + B'_{1-Y_2} \ge B_{Y_3} + B'_{1-Y_3}$ with probability $\frac{1}{2}$ conditioned on $Y_2$ and $Y_3$ (so the randomness is from the algorithmic randomness). Therefore, $\left(P\middle\vert Z\right)(\cdot) \sim \frac{1}{2}\delta_{\frac{1}{3}} + \frac{1}{2}\delta_{1}$ almost surely in $Z$ (not in the algorithmic random space).
    Thus, almost surely in $Z$,
    \begin{align*}
        \prob{P\le \frac{1}{3}} = 
        \expect{\cprob{P\le \frac{1}{3}}{Z}} = \frac{1}{2} > \frac{1}{3},
    \end{align*}
    contradicting with that $\mathbb{P}\left(P \le \frac{1}{3}\right) \le \frac{1}{3}$ for $P$ to be a valid p-value. Thus, $P$ cannot be a valid p-value.

    Thus, the sufficiency claims in~\citep{barber2023conformal,bian2023training,lee2023distribution,bai2024optimized} are not correct.\footnote{It is easier to construct a counter example when we allow the random function for each score computation to use different part of the randomness. For example, we can consider the case when $\omega \sim \text{Unif}([0,1]^2)$, and $S_1=S_2=\omega_1$, while $S_3 = \omega_2$.}

\section{Examples of sufficient conditions}

    Recall that the validity of conformal prediction comes from the validity of permutation test. Thus, we only need that, as long as the data points $Z$ are exchangeable,
    the non-conformity scores $S := \cc{S_i}_{i=1}^{N+1}$ are also exchangeable (unconditionally). 

    \subsection{Deterministic with Permutation Invariance}
    \label{subsection:deterministic}

    For deterministic case, a simple sufficient condition is permutation invariant among first $N$ arguments, meaning that
    \begin{align}
    \label{eqn:det}
        T_i(z;\omega) = T_j(\cc{z_{\sigma(k)}}_{k=1}^N, z_{N+1}; \omega')
    \end{align}
    for all $i,j \in \{1,2,\dots, N+1\}$, $z \in \mathcal{Z}^{N+1}$, $\sigma \in S_N$, and for almost surely $\omega, \omega'$. This is a standard condition, and can be achieved by symmetrization over a deterministic function $f: \mathcal{Z}^{N+1} \to \mathbb{R}$ such as choosing $T_i(z;\omega) = \frac{1}{N!} \sum_{\sigma \in S_N} f\left(z_{\sigma(1)}, z_{\sigma(2)}, \dots, z_{\sigma(N)}, z_{N+1}\right)$, but this requires $N!$ computation.

    \subsection{Stochastic with Almost Sure Permutation Invariance}
    \label{subsection:almostsure}
    As mentioned, this is when we weaken the degeneracy requirement (by dropping $\omega'$) from the condition~\ref{eqn:det} into
    \begin{align}
    \label{eqn:ass}
        T_i(z;\omega) = T_j(\cc{z_{\sigma(k)}}_{k=1}^N, z_{N+1}; \omega)
    \end{align}
    almost surely. This is still hard to achieve, as it is translated into that the permutation invariance holds for each random seed $\omega$, and we have to reuse the same seed for each score evaluation process. For example, this can be satisfied when the randomness only comes in through random initialization of neural network parameters and full gradient descent is implemented. Another example is when symmetrization is used with the same fixed seed.
    
    For most machine learning training, this rarely holds. Stochastic gradient descent with mini-batch\footnote{Unless we enforce the algorithm to re-sort the datapoints first before using any stochasticity.} may sample a random index $I(\omega)$ of the data to be used in the gradient computation. By applying the permutation before running the algorithm, the $I(\omega)^{\text{th}}$ datapoint in the permuted dataset will correspond to the $\sigma\left(I(\omega)\right)^{\text{th}}$ datapoint of the original dataset.

    Moreover, this condition is difficult to satisfy if we have randomness in the inference.

    \subsection{Stochastic with Independence \& Permutation Invariance in Distribution}
    \label{subsection:conditionalindependence}

    We weaken the almost sure equality~\ref{eqn:ass} to equality in distribution
    \begin{align}
    \label{eqn:indsym}
        T_i\oo{z; \cdot}
        \overset{\text{d}}{=}
        T_j\oo{\cc{z_{\sigma(k)}}_{k=1}^N, z_{N+1}; \cdot},
    \end{align}
    but with additional independence condition that
    \begin{align*}
        \left\{T_i\oo{z_1, z_2, \dots, z_{i-1}, z_{i+1}, z_{i+2}, \dots, z_{N+1}, z_i;\cdot}\right\}_{i=1}^{N+1}
    \end{align*}
    is a collection of independent variables.

    This is perhaps the most commonly used method in full conformal prediction. Each time we train a machine learning model, the stochasticity is assumed to be independent from those in the past. If there exists randomness in the inference, as long as such randomness are independent from each other, then the validity of full conformal prediction will still held. However, we can see that this can be violated in some case, such as when each training share the same random seed, which may be used for reproducibility.  

\section{General Sufficient Condition: Conditional Independence \& Permutation Invariance in Distribution}

    Note that the deterministic case~\ref{subsection:deterministic} is a special case for both stochastic cases~\ref{subsection:almostsure} and~\ref{subsection:conditionalindependence}, but the two stochastic cases are not a special case of each other. We generalize the combination of the two stochastic cases to show that a more general condition of conditional independence and permutation invariance in distribution is sufficient for full conformal prediction validity.

    \begin{definition}
    [Conditional Independence \& Permutation Invariance in Distribution]
        \label{definition:CIS}
            The stochastic non-conformity score functions $T=[T_i]_{i=1}^{N+1}$, where $T_i: \mathcal{Z}^{N+1}\times \Omega \to \mathcal{S}$ for each $i \in \{1,2,\dots, N+1\}$ are conditionally independent and permutation invariant in distribution if there exists some random variable $C$ defined on the algorithmic random space such that, for all $z \in \mathcal{Z}^{N+1}$ and for almost surely $C$,
            \begin{align}
            \label{eqn:condsym}
                \left(T_i\oo{z; \cdot} \middle\vert C\right)
                \overset{\text{d}}{=}
                \left(T_j\oo{\cc{z_{\sigma(k)}}_{k=1}^N, z_{N+1}; \cdot} \middle\vert C\right)
            \end{align}
            for all $\sigma \in S_N$, for all $i,j \in \{1,2,\dots, N+1\}$, and the collection of random variables
            \begin{align*}
                \left\{
                    T_i\oo{z_1, z_2, \dots, z_{i-1}, z_{i+1}, z_{i+2}, \dots, z_{N+1}, z_i;\cdot}
                \right\}_{i=1}^{N+1}
            \end{align*}
            is independent conditioned on $C$.
        \end{definition}

        By taking $C$ to be non-informative, this reverts back to the independence and permutation invariance in distribution setting~\ref{subsection:conditionalindependence}, and, by taking $C$ to be $\omega$, this reverts back to the almost sure permutation invariance~\ref{subsection:almostsure}.

        \begin{theorem}
        \label{theorem:validity}
            The conformal p-value whose stochastic non-conformity score functions $T$ are conditionally independent and permutation invariant in distribution is valid.
        \end{theorem}
        The proof\footnote{which is in the appendix~\ref{appendix:proof}}, first, shows that the joint score computation, with
        $z_{-i}:=(z_1,\ldots,z_{i-1},z_{i+1},\ldots,z_{N+1})$,
        \begin{align*}
            A(z;\omega) := \cc{T_i\oo{z_{-i},z_i;\omega}}_{i=1}^{N+1}
        \end{align*}
    is permutation equivariance in joint distribution. This means that
        \begin{align*}
            \cc{\oo{A(z; \cdot)}_{\pi(i)}}_{i=1}^{N+1} \overset{\text{d}}{=} A(\cc{z_{\pi(i)}}_{i=1}^{N+1};\cdot)
        \end{align*}
        for all $z \in \mathcal{Z}^{N+1}$ and $\pi \in S_{N+1}$.

        Afterwards, we use exchangeability of $Z$ to conclude that the scores
        \begin{align*}
            S =A(Z;\omega)
        \end{align*}
        are also exchangeable.

\begin{figure}[t]
    \centering
    \includegraphics[width=1\linewidth]{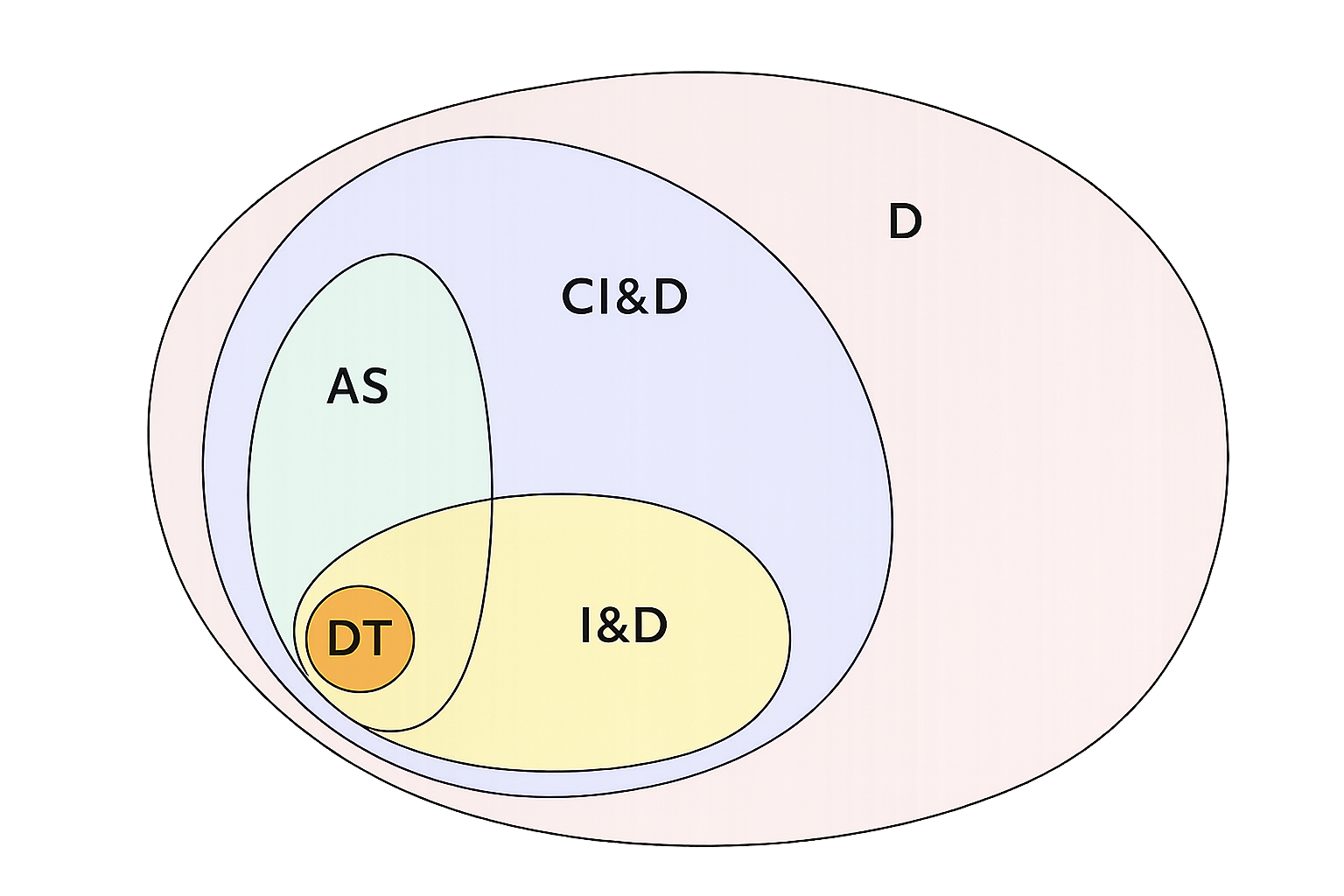}
    \caption{\emph{DT} stands for deterministic permutation invariance, which is a standard assumption in full conformal prediction as outlined in the subsection~\ref{subsection:deterministic}.~\citep{vovk2005algorithmic} \emph{AS} stands for almost sure permutation invariance as outlined in the subsection~\ref{subsection:almostsure} \emph{I\&D} stands for independence and permutation invariance in distribution as outlined in the subsection~\ref{subsection:conditionalindependence}. \emph{CI\&D} stands for conditional independence and permutation invariance in distribution as defined in the definition~\ref{definition:CIS}. We have established through the main theorem~\ref{theorem:validity} that \emph{CI\&D} is sufficient to ensure the validity of full conformal prediction when the non-conformity measure is stochastic. \emph{D} stands for permutation invariance in distribution, which is the incorrect sufficient condition in the literature~\citep{barber2023conformal,bian2023training,lee2025full,bai2024optimized}. In this figure, we show that $\text{DT} \subseteq \text{AS} \cap \text{I\&D}$, while $\text{AS} \cup \text{I\&D} \subseteq \text{CI\&D} \subseteq \text{D}$. Note that these conditions are defined with respect to an algorithmic random space $\oo{\Omega, \mathcal{F}, \mathbb{P}_{\text{alg}}}$ as well as $\mathcal{S}$ and $\mathcal{Z}$. It is easy to show that, for each $\subseteq$ relation, we can find a setting such that every $\subseteq$ relation is indeed $\subsetneq$.}
    \label{fig:placeholder}
\end{figure}

    This condition is general enough that could be used to satisfied most valid use case of full conformal prediction. For example, one can imagine when the random seed for the random initialization of the neural network parameters are fixed, while the random seeds for the order of stochastic gradient descent is redrawn each time. In such case, the condition $C$ will represent the initialization of the neural network parameters.

    It is also easy to see that this is not a necessary condition. For example, we can replace with the condition like conditional exchangeability. This condition will imply the condition ``Conditional Independence \& Permutation Invariance in Distribution" we introduced but at the cost of harder verification.

    Another sufficient condition will be to revert to only use deterministic score function/measure but augmenting the inputs $z$ with white noise.~\citep{vovk2005algorithmic} It can be seen that this is not implying nor implied by our condition. This is because the conditional independence requirement is only ``on the path", and does not require every possible permutations in a two-dimensional table to satisfy the conditional independence.
    
\section{Conclusion}

    Our paper shows that the long incorrectly believed assumption for the validity of full conformal prediction is however incorrect. We show it by a simple construction with a generous interpretation of the claims made in the literature. We then provide a sufficient condition.

\section*{Acknowledgement}

    The author thanks Lihua Lei and Vladimir Vovk for useful discussions.

\bibliographystyle{icml2026}
\bibliography{egbib}

\appendix

\section{p-Value \& permutation test}
\label{appendix:pperm}

    We define a valid p-value, and provide background on permutation test, which is the backbone conformal prediction.
    
    \begin{definition}
        [Valid p-Value]
        A random variable $P$ is a valid p-value if
        \begin{align*}
            \prob{P \le \alpha} \le \alpha 
        \end{align*}
        for all $\alpha \ge 0$.
    \end{definition}
    
    \begin{lemma}
    \label{lemma:exchangepvalue}
        Let $S:=\cc{S_i}_{i=1}^n$ be an exchangeable random variable where $S_i \in \mathcal{S}$ for each $i \in \{1,2,\dots, n\}$ and $\oo{\mathcal{S}, \succcurlyeq}$ is a total preorder. Define
        \begin{align*}
            P_1 := \frac{\sum_{i=1}^n \mathbf{1}_{S_1 \succcurlyeq S_i}}{n}.
        \end{align*}
        Then its distribution satisfies
        \begin{align*}
            \mathcal{L}(P_1)
            \succcurlyeq_{\mathrm{FOSD}}
            \mathrm{Unif}\!\left(\left\{\frac{1}{n},\ldots,1\right\}\right)
            \succcurlyeq_{\mathrm{FOSD}}
            \mathrm{Unif}([0,1]).
        \end{align*}
    \end{lemma}
    \begin{proof}
        Define $P_i:= \frac{\sum_{j=1}^n \mathbf{1}_{S_i \succcurlyeq S_j}}{n}$. Exchangeability of $S$ implies exchangeability of $P$. Note that, for any integer $k \in \{1,2,\dots, n\}$, there can be at most $k$ values of $i \in \{1,2,\dots, n\}$ where $P_i \le \frac{k}{n}$. Hence,
        \begin{align*}
            \prob{P_1 \le k/n}
            &= \expect{\frac{1}{n}
            \sum_{i=1}^n \mathbf{1}_{P_i \le k/n}} \\
            &\le \frac{k}{n},
        \end{align*}
        implying the first order stochastic dominations.
    \end{proof}
    \begin{corollary}
    \label{corollary:validpexchange}
        $P_1$ in the lemma~\ref{lemma:exchangepvalue} is a valid p-value if $S$ is exchangeable.
    \end{corollary}
    \begin{proof}
        From the lemma~\ref{lemma:exchangepvalue}, we have that $\mathbb{P}\left(P_1 \le \alpha\right) \le \left(\text{Unif}([0,1])\right)([0, \alpha]) = \alpha$, so $P_1$ is a valid p-value.
    \end{proof}

    In permutation testing~\citep{lehmann2005testing}, we test a hypothesis of whether random variables $Z := \cc{Z_i}_{i=1}^n$ are exchangeable, by randomly permuting the variables and compute test statistic each time. Under such hypothesis, the test statistics should be exchangeable, so $P_1 \le \alpha$ with probability $\le \alpha$, meaning that it is a valid p-value.  

\section{Permutation Equivariance in Distribution algorithm}
\label{appendix:Equi}

    To prove the main theorem~\ref{theorem:validity}, we recall that $Z$ is an exchangeable random variable whose stochasticity is independent from that of the stochastic algorithm and $\oo{\mathcal{S},\succcurlyeq}$ is already assumed to be a total preorder. Thus, we only need to ensure that a stochastic algorithm $A: \mathcal{Z}^{N+1} \times \Omega \to \mathcal{S}^{N+1}$ that jointly return the score
    \begin{align*}
        S := \cc{S_i}_{i=1}^{N+1} := A(Z; \omega)
    \end{align*}
    will be able to transfer the exchangeability in $Z$ into the exchangeability in $S$, both of which in a distribution sense. We can then construct a p-value
    \begin{align*}
        P := \frac{1}{N+1}\cc{1+\sum_{i=1}^N \mathbf{1}_{S_{N+1} \succcurlyeq S_i}},
    \end{align*}
    and get its validity by the corollary~\ref{corollary:validpexchange}.

    In this section, we first show that permutation equivariant in joint distribution is a sufficient condition (in the definition~\ref{def:equivaraibt}) for the stochastic algorithm $A$ to preserve exchangeability. Next, we will show a condition on the non-conformity score functions $T$ to achieve permutation equivariant in joint distribution.

    \subsection{Definition \& property}
    
        A convenient sufficient condition for exchangeability to be inherited through $A$ is that the stochastic algorithm $A$ is permutation equivariant in joint distribution.
        \begin{definition}
            [Permutation Equivariance in Joint Distribution]
            \label{def:equivaraibt}
            A random function $A: \mathcal{Z}^{N+1} \times \Omega \to \mathcal{S}^{N+1}$ is permutation equivariant in joint distribution if, for any permutation $\pi \in S_{N+1}$, for any $z \in \mathcal{Z}^{N+1}$,
            \begin{align*}
                A\left(\left[z_{\pi(i)}\right]_{i=1}^{N+1}; \cdot\right)
                \overset{\text{d}}{=}
                \left[
                (A(z;\cdot))_{\pi(i)}
                \right]_{i=1}^{N+1}.
            \end{align*}
        \end{definition}
        Recall that exchangeable is the same as permutation invariance in distribution. It is then unsurprising that $S = A(Z;\omega)$ can inherit the exchangeability\footnote{Note that the distribution of $Z$ is fixed to be an exchangeable one, so we leave it out of the lemma statement.} in distribution from $Z$.
        \begin{lemma}
        [Exchangeability \& Permutation Equivariance imply Exchangeability]
        \label{lemma:exchange+equivariance}
            If $A: \mathcal{Z}^{N+1} \times \Omega \to \mathcal{S}^{N+1}$ is permutation equivariant in joint distribution, then $A(Z; \omega)$ is exchangeable.
        \end{lemma}
        \begin{proof}
            Consider a permutation $\pi \in S_{N+1}$. From exchangeability of $Z$, and the permutation equivariance of $A$, we have that
            \begin{align*}
                A(Z) \overset{\text{d}}{=}
                A\oo{\cc{Z_{\pi(i)}}_{i=1}^{N+1}}
                \overset{\text{d}}{=}
                \cc{\oo{A(Z)}_{\pi(i)}}_{i=1}^{N+1},
            \end{align*}
            meaning that it is exchangeable.
        \end{proof}

        Therefore, we get the following corollary.
        \begin{corollary}
        \label{corollary:validpfromA}
            If $A: \mathcal{Z}^{N+1} \times \Omega \to \mathcal{S}^{N+1}$ is permutation equivariant in joint distribution, and $S = A(Z; \omega)$, then, by defining 
            \begin{align*}
                P := \frac{1}{N+1}\cc{1+\sum_{i=1}^N \mathbf{1}_{S_{N+1} \succcurlyeq S_i}},
            \end{align*}
            we will have that $P$ is a valid p-value.
        \end{corollary}
        \begin{proof}
            From the lemma~\ref{lemma:exchange+equivariance}, we have that $S$ is exchangeable. Note that the construction of $P$ is the same (but use the last index instead of first index) as $P_1$ in the lemma~\ref{lemma:exchangepvalue}, so by the corollary~\ref{corollary:validpexchange}, $P$ is a valid p-value.
        \end{proof}

    \subsection{Construction \& sufficient condition}

        In the last subsection, we have established that permutation equivariance will be sufficient for the validity of the rank p-value. Note that this is more general than a conformal p-value, since $A$ does not have to create non-conformity scores like in the algorithm~\ref{alg:stochp}. In this subsection, we will restrict the attention when the algorithm $A$ generates non-conformity scores under full conformal prediction framework. For any vector $u=(u_1,\ldots,u_{N+1})$, write
        $u_{-i}:=(u_1,\ldots,u_{i-1},u_{i+1},\ldots,u_{N+1})$.
        This is when 
        \begin{align*}
            A\left(z; \omega\right)
            :=
            \left[
                T_i\left(z_{-i}, z_i; \omega\right)
            \right]_{i=1}^{N+1}
        \end{align*}
        for any $z \in \mathcal{Z}^{N+1}$.

        Consider a permutation $\pi\in S_{N+1}$ and write $z^\pi:=\cc{z_{\pi(i)}}_{i=1}^{N+1}$. For each coordinate $i$, we have
        \begin{align*}
            \oo{A\left(z^\pi; \omega\right)}_i
            &= T_i\left(\oo{z^\pi}_{-i},z_{\pi(i)};\omega\right),\\
            \oo{A\left(z; \omega\right)}_{\pi(i)}
            &= T_{\pi(i)}\left(z_{-\pi(i)},z_{\pi(i)};\omega\right).
        \end{align*}
        Thus, for each coordinate, the two random vectors share the same last element, but with different ordering and index of the score function. If we want permutation equivariance in joint distribution, we then need the marginal distribution of each coordinate to be the same.
        This can be easily satisfied by enforcing permutation invariance in distribution~\ref{eqn:indsym}. However, as we have shown in the section~\ref{section:distributionisnotdufficiebt}, this is not sufficient, since this will only give permutation equivariance in marginal distributions but not joint distribution. Thus, a stronger condition is required.

        For example, if we have almost sure permutation invariance~\ref{eqn:ass}, we can rearrange the first $N$ terms and change the index from $T_i$ to $T_{\pi(i)}$ for each coordinate. This gives $A\left(\left[z_{\pi(i)}\right]_{i=1}^{N+1}; \omega\right) = \left[(A(z; \omega))_{\pi(i)}\right]_{i=1}^{N+1}$ almost surely, which is stronger than what required in permutation equivariance in joint distribution.

        Next, we will show that conditional independence and permutation invariance in distribution condition of the stochastic scoring functions $T$ ensures that the stochastic algorithm $A$ is permutation equivariant in joint distribution.
        \begin{lemma}
        [Conditional Independence \& Permutation Invariance in Distribution implies Permutation Equivariance]
        \label{lemma:CISthenequivaraint}
            If the stochastic non-conformity score functions $T$ are conditionally independent and symmetric, then the stochastic algorithm
            $A\left(z; \omega\right):=\left[T_i\left(z_{-i},z_i;\omega\right)\right]_{i=1}^{N+1}$ almost surely for all $z \in \mathcal{Z}^{N+1}$ is permutation equivariant in joint distribution.
        \end{lemma}
        \begin{proof}
            From Dynkin's $\pi-\lambda$ theorem, it is sufficient to show the equality in probability over rectangular set. Consider an arbitrary $B := \bigotimes_{i=1}^{N+1} B_i$, where each $B_i$ is a measurable subset of $\mathcal{S}$. Write $z^\pi:=\cc{z_{\pi(i)}}_{i=1}^{N+1}$. From the definition~\ref{definition:CIS} and the definition of $A$, we have that the collection
            \begin{align*}
                \left\{
                    \oo{A\oo{z;\cdot}}_i
                \right\}_{i=1}^{N+1}
            \end{align*}
            is independent conditioned on $C$, so
            \begin{align*}
                \cprob{A\left(z^\pi\right) \in B}{C}
                &=
                \cprob{\oo{A\left(z^\pi\right)}_i \in B_i;
                \forall i}{C}\\
                &=
                \prod_{i=1}^{N+1}
                \cprob{\oo{A\left(z^\pi\right)}_i \in B_i}{C}.
            \end{align*}
            Note that, from the conditional symmetry in~\ref{eqn:condsym},
            \begin{align*}
                \left(\oo{A\left(z^\pi\right)}_i\middle\vert C\right)
                &=
                \left(
                T_i\oo{\oo{z^\pi}_{-i},z_{\pi(i)}}
                \middle\vert C\right)\\
                &\overset{\text{d}}{=}
                \left(
                T_{\pi(i)}\oo{z_{-\pi(i)},z_{\pi(i)}}
                \middle\vert C\right)\\
                &=
                \left(
                \oo{A(z)}_{\pi(i)}
                \middle\vert C\right),
            \end{align*}
            so
            \begin{align*}
                \cprob{A\left(z^\pi\right) \in B}{C}
                &=
                \prod_{i=1}^{N+1}
                \cprob{(A(z))_{\pi(i)} \in B_i}{C}\\
                &=
                \cprob{(A(z))_{\pi(i)} \in B_i;
                \forall i}{C}\\
                &=
                \cprob{\cc{(A(z))_{\pi(i)}}_{i=1}^{N+1} \in B}{C}.
            \end{align*}
            Note that all equalities here are held almost surely $C$. We then take expectation, so
            \begin{align*}
                \prob{A\left(z^\pi\right) \in B}
                =\prob{\cc{(A(z))_{\pi(i)}}_{i=1}^{N+1} \in B}.
            \end{align*}
            Thus, $A$ is permutation equivariant in joint distribution.
        \end{proof}

\section{Proof of the main theorem}
\label{appendix:proof}
    Now that we have proved other claims, we can put them together to prove the main theorem.
    \begin{proof}
         From the lemma~\ref{lemma:CISthenequivaraint}, we have that the joint score function is permutation invariant in joint distribution. From the corollary~\ref{corollary:validpfromA}, we then have that $P$ is a valid p-value.
    \end{proof}

\section{Strict relationships}

As mentioned in the figure~\ref{fig:placeholder}, the relationships
\begin{align*}
    \text{DT} \subseteq \text{AS} \cap \text{I\&D},
\end{align*}
and 
\begin{align*}
    \text{AS} \cup \text{I\&D} \subseteq \text{CI\&D} \subseteq \text{D}
\end{align*}
can be strengthened from $\subseteq$ to $\subsetneq$ for some appropriate setting $\oo{\oo{\Omega, \mathcal{F}, \mathbb{P}_{\text{alg}}}, \mathcal{Z}, \mathcal{S}}$.

Consider the case when $N = 3$, $\Omega = [0,1]^2$, $\mathcal{F}$ is the Borel $\sigma$-algebra, $\mathbb{P}_{\text{alg}}$ is the Lebesgue measure, $\mathcal{Z}$ is $\{0,1\} \times [0,1]$, and $\mathcal{S} = \mathbb{R}$.

\subsection{$\text{CI\&D} \subsetneq \text{D}$}

    Since the random space is sufficiently rich to define $2$ independent Brownian motion, the same construction in the section~\ref{section:distributionisnotdufficiebt}\footnote{The construction in the section~\ref{section:distributionisnotdufficiebt} assumes $\mathcal{X} = [0,1]$, but $\mathcal{X}$ is not used in the definition, so the result will be the same.} will provide the example of $T$ that is permutation invariance in distribution but does not yield a valid conformal p-value. From the theorem~\ref{theorem:validity}, we can then conclude that $T$ is not conditionally independence and permutation invariance in distribution. Therefore,
    \begin{align*}
        \text{CI\&D} \subsetneq \text{D}.
    \end{align*}
    
\subsection{$\text{DT} \subsetneq \text{AS} \cap \text{I\&D}$}

    We define
    \begin{align*}
        T_i(z;\omega) :=  \begin{cases}
            \omega_1 &\text{ if } x_1=x_2=0 \text{ and } x_3=1\\
            \omega_2 &\text{ if } x_1=x_2=1 \text{ and } x_3=0\\
            0 &\text{ otherwise}
        \end{cases}
    \end{align*}
    for all $i\in \{1,2,3\}$, $z \in \mathcal{Z}^3$, and $\omega \in \Omega$.

    Thus, $T$ is not deterministic, but it is almost surely permutation invariant.\footnote{Recall that we only require the permutation invariance over the first $N=2$ arguments.} 
    
    Next, we will show that this is independent and permutation invariant in distribution. As we have discussed, permutation invariant in distribution is weaker than the almost sure condition, so it is automatically satisfied. Since, among $T_1(z_2,z_3,z_1;\omega)$, $T_2(z_1,z_3,z_2;\omega)$, and $T_3(z_1,z_3,z_2;\omega)$, only at most one can be non-zero, we then get independence. Therefore,
    \begin{align*}
        \text{DT} \subsetneq \text{AS} \cap \text{I\&D}.
    \end{align*}

\end{document}